\documentclass
[
11pt]
{amsart}%

\usepackage{hyperref}
\usepackage{cite}
\usepackage{amssymb}
\usepackage{amsmath}
\usepackage{amsthm}
\usepackage{amsfonts}
\usepackage{graphicx}
\usepackage{bbm}
\usepackage{xcolor}
\usepackage[toc,page]{appendix}
\usepackage{subfigure}
\usepackage{mathtools}
\usepackage{bm}
\usepackage{comment}

\newcommand{\beq}{\begin{equation}}
\newcommand{\eeq}{\end{equation}}
\newcommand{\bea}{\begin{eqnarray}}
\newcommand{\eea}{\end{eqnarray}}
\newcommand{\beas}{\begin{eqnarray*}}
\newcommand{\eeas}{\end{eqnarray*}}

\newtheorem{theorem}{Theorem}[section]

\newtheorem{assumption}[theorem]{Assumption}

\newtheorem{proposition}[theorem]{Proposition}
\newtheorem{prop}[theorem]{Proposition}

\newtheorem{lemma}[theorem]{Lemma}
\newtheorem{remark}[theorem]{Remark}
\newtheorem{example}[theorem]{Example}
\newtheorem{examples}[theorem]{Examples}
\newtheorem{foo}[theorem]{Remarks}

\newtheorem{notation}[theorem]{Notation}

\newcommand{\var}[1]{{\rm Var}\left(#1\right)}

\newcommand{\p}{\partial}

\newcommand{\R}{\mathbb R}

\newcommand{\masha}[1]{}
\newcommand{\fabrice}[1]{}
\newcommand{\tai}[1]{}

\begin{document}

\title[Quasi-invariance for Kolmogorov diffusions]{Quasi-invariance for infinite-dimensional Kolmogorov diffusions}

\author[Baudoin]{Fabrice Baudoin{$^{\ast}$}}
\thanks{\footnotemark {$\ast$} This research was supported in part by NSF Grant DMS-1901315.}
\address{Department of Mathematics\\
University of Connecticut\\
Storrs, CT 06269, USA} \email{fabrice.baudoin@uconn.edu}

\author[Gordina]{Maria Gordina{$^{\ast\ast}$}}
\thanks{\footnotemark {$\ast\ast$} This research was supported in part by NSF Grant DMS-1954264.}
\address{Department of Mathematics\\
University of Connecticut\\
Storrs, CT 06269, USA} \email{maria.gordina@uconn.edu}

\author[Melcher]{Tai Melcher{$^{\dagger}$}}
\thanks{\footnotemark {$\dagger$} This research was supported in part by NSF Grant DMS-1255574.}
\address{Department of Mathematics\\
University of Virginia \\
Charlottesville, VA 22903, USA} \email{melcher@virginia.edu}

\date{\today \ \emph{File:\jobname{.tex}}}

\begin{abstract}
We prove Cameron-Martin type quasi-invariance results for the heat kernel measure of infinite-dimensional Kolmogorov and related diffusions. We first study quantitative functional inequalities for appropriate finite-dimensional approximations of these diffusions, and we prove these inequalities hold with dimension-independent coefficients. Applying an approach developed in \cite{BaudoinGordinaMelcher2013,DriverGordina2009, Gordina2017}, these uniform bounds may then be used to prove that the heat kernel measure for certain of these infinite-dimensional diffusions is quasi-invariant under changes of the initial state.
\end{abstract}

\maketitle
\tableofcontents


\renewcommand{\contentsname}{Table of Contents}

\section{Introduction}

Smoothness properties are of classical interest in the study of measures in infinite dimensions. These properties have been a particular focus for measures associated to diffusions, either the measure induced on the path space or the end point distribution (that is, the heat kernel measure) for diffusions taking values in an infinite-dimensional space. In finite dimensions, hypoellipticity of the generator is a standard hypothesis to assume for (or is in some sense equivalent to) regularity properties of the heat kernel measure. It is natural to consider analogues of various hypoelliptic constructions in infinite dimensions, keeping in mind that even the definition of hypoellipticity is not easily available in such a setting.

In \cite{Hormander1967a} L.~H\"ormander put forward $L=\frac{1}{2}\frac{\partial^2}{\partial p^2} + p \frac{\p}{\p\xi}$ on $\R^2$ as the simplest example of a hypoelliptic second order differential operator. This operator generates the standard Kolmogorov diffusion on $\R^2$, a Gaussian process which may be written as a standard real-valued Brownian motion paired with its integral in time
\[
X_t = \left(B_t,\int_0^t B_s\,ds\right).
\]
The operator $L$ had previously been introduced by A.N.\,Kolmogorov in \cite{Kolmogorov1934}, in which he obtained the explicit expression for the (smooth) density of $\mathrm{Law}(X_t)$ with respect to Lebesgue measure
\[
p_t(p,\xi) = \frac{\sqrt{3}}{\pi t^2} \exp\left(-\frac{2p^2}{t}+\frac{6p\xi}{t^2}-\frac{6\xi^2}{t^3}\right).
\]
Even in the finite-dimensional setting, despite having an explicit Gaussian heat kernel, it is challenging to derive typical functional inequalities for the associated semigroup.

In the present paper, we consider a natural infinite-dimensional analogue of the Kolmogorov diffusion as well as a variation of that construction. Note that in the infinite-dimensional setting we can not rely on standard definitions of hypoellipticity. In particular, the heat kernel measure is not a measure with a density with respect to Lebesgue measure anymore. Still we can prove quasi-invariance results to reflect regularity of the heat kernel measure. More specifically, let
$\left( W, H, \mu \right)$ be an abstract Wiener space, and $\{B_t\}_{t \geqslant  0}$ be a Brownian motion on $W$ with covariance structure determined by
\[
\mathbb{E}\left[\langle B_s, h\rangle_H \langle B_t, k\rangle_H\right]
    = \langle h, k \rangle_H \min(s,t),
\]
for all $s,t \geqslant  0$ and $h, k \in W^{\ast}\subset H$. For more details on the setting  we refer to Section \ref{s.infinite}. We define an \emph{infinite-dimensional Kolmogorov diffusion} as the $W \times W$-valued process given by
\[
X_t=\left(B_t , \int_0^t B_s ds \right).
\]
The process $X_t$ is Gaussian with the covariance determined by
\begin{align*}
 & \mathbb{E}\left[\langle X_s,h_1 \otimes h_2 \rangle_{H\otimes H} \langle X_t,k_1 \otimes k_2 \rangle_{H \otimes H}\right] \\
    = &  \langle h_1,k_1 \rangle_H \min(s,t) +\langle h_1,k_2 \rangle_H\int_0^t \min(v,s)dv  +\langle k_1,h_2 \rangle_H \int_0^s \min (u,t) du \\
     & +\langle h_2,k_2 \rangle_H \int_0^t \int_0^s \min(u,v)\, du\,dv  ,
\end{align*}
for all $s,t \geqslant  0$ and $h_1,k_1,h_2,k_2\in W^*$.

For $x\in W\times W$, let $X_t^x$ denote a Kolmogorov diffusion started at $x$, and let $\nu_t^x$ denote its distribution, with $\nu_t^0=:\nu_t$. We will prove that for $ t >0 $ the measure $\nu_t^x$ is equivalent to $\nu_t$ when $x\in H\times H$. That is, for $x=(h,k)$ we have
\[
X_t^{h,k}=X_t+\left(h,k +t h \right),
\]
and we will show that $\nu_t^{(h,k)}=\operatorname{Law}(X_t^{(h,k)})$ is mutually absolutely continuous with $\nu_t$ when $(h,k)\in H\times H$. We will also give bounds on the $L^p$ norms of the associated Radon-Nikodym derivatives. We then consider an analogous non-Gaussian process
\begin{equation}\label{e.KolmogrovDiff}
X_t=\left(B_t , \int_0^t F(B_s) ds \right)
\end{equation}
for functions $F$ taking values in finite- or infinite-dimensional spaces, and we prove similar results in this case under the assumption of certain derivative bounds for $F$.

Before further describing the structure of the paper and the main results, let us mention several papers most relevant to the techniques we are using. In \cite{BaudoinGordinaMariano2020} the first two authors and Ph.~Mariano study gradient bounds and other related functional inequalities for Kolmogorov-type diffusions in finite dimensions via both the generalized $\Gamma$-calculus techniques employed in the present paper as well as through coupling methods. The method we follow for proving quasi-invariance for measures in infinite dimensions via functional inequalities was developed in \cite{BaudoinGordinaMelcher2013,DriverGordina2009, Gordina2017}. This approach relies on certain inequalities holding with dimension-independent coefficients on appropriate finite-dimensional approximations. These methods were successfully applied to infinite-dimensional Heisenberg groups in \cite{DriverGordina2008} in the elliptic case and \cite{BaudoinGordinaMelcher2013} in the hypoelliptic case. A different proof for the hypoelliptic Heisenberg setting was given in \cite{DriverEldredgeMelcher2016}; in that paper, it was in fact proved that the hypoelliptic heat kernel measure satisfied stronger smoothness criteria. The paper \cite{BaudoinFengGordina2019} gives another example of using geometric methods to prove quasi-invariance for path spaces over totally geodesic Riemannian foliations equipped with a sub-Riemannian structure.

Like a diffusion in the Heisenberg group, the Kolmogorov diffusion stands as one of the simplest examples of a hypoelliptic diffusion. Still, given the obstacles which are often unique to each construction of infinite-dimensional diffusions, the proof of dimension-independent estimates typically requires ad hoc methods. In particular, the Kolmogorov diffusion provides an example of an operator satisfying a weak H\"ormander's condition, where the drift is required to generate the full span of the tangent space. This means that techniques recently developed to study diffusions in sub-Riemannian manifolds are not directly applicable.  Also, in the case of the generalized Kolmogorov diffusions we consider, there is no underlying group structure and so a shift in the initial state of the process does not correspond to any group action. Thus the results in the current paper can not be proven by directly applying  the approach developed in \cite{DriverGordina2009, Gordina2017}.

It is additionally worth noting that in this paper we consider a non-Gaussian generalization of the Kolmogorov diffusion, to demonstrate the application of the method to a non-Gaussian diffusion.
Thus, the method of proof we use is robust and not restricted to Gaussian situations.

The paper is organized as follows. In Section \ref{s.finite}, we establish the necessary functional inequalities in the finite-dimensional settings. In Section \ref{s.infinite}, we construct the processes in infinite dimensions and use results proved in Section \ref{s.finite} to demonstrate the existence of a Radon-Nikodym derivative for the ``shifted'' heat kernel measure with respect to the heat kernel measure started at $0$. Theorem \ref{t.kqi} proves quasi-invariance for the ``standard'' infinite-dimensional Kolmogorov diffusion. Remark \ref{r.CMM} shows how one can prove the quasi-invariance result in this case as an application of the standard Cameron-Martin-Maruyama theorem on path space.
 Theorem \ref{t.gkqi3} gives conditions under which quasi-invariance holds for generalized Kolmogorov diffusions of the form (\ref{e.KolmogrovDiff}) where $F$ takes values in a finite-dimensional vector space. Theorem \ref{t.gkqi4} treats the case where $F$ takes values in $W$. Example \ref{ex.genex} gives some examples of where Theorem \ref{t.gkqi4} may be applied, and revisits the standard case, showing that one may actually get a better upper bound on the Radon-Nikodym derivative than the estimate proved in Theorem \ref{t.kqi}.

\textit{Acknowledgement.} The authors are grateful to Nate Eldredge for several helpful discussions during the writing of this paper.

\section{Wang's Harnack inequality for finite-dimensional diffusions}
\label{s.finite}

We first establish the necessary functional inequalities for diffusions taking values in finite dimensions. We initially treat the case of the standard Kolmogorov diffusion before considering a non-Gaussian generalization of this construction.

\subsection{Kolmogorov diffusions}
\label{s.fk}

Let $X_t=(B_t,\int_0^tB_s ds)$ denote the finite-dimensional Kolmogorov diffusion in $\mathbb{R}^d \times \mathbb{R}^d$, where $B$ is a standard Brownian motion on $\mathbb{R}^d$. We consider functions $f=f \left( p, \xi\right)$ on $\mathbb{R}^d \times \mathbb{R}^{d}$. For $f \in C^{2}\left( \mathbb{R}^d \times \mathbb{R}^{d} \right)$, let
\begin{align*}
	\left( Lf \right) \left( p, \xi \right) &:= \langle p, \nabla_\xi f \left( p, \xi\right) \rangle+\frac{1}{2} \Delta_pf \left( p, \xi \right) \\
	&= \sum_{j=1}^{d} p_{j}\frac{\partial f}{\partial \xi_{j}}\left( p, \xi \right)+\frac{1}{2} \sum_{j=1}^{d} \frac{\partial^2f}{\partial p_j^2} \left( p, \xi \right),
\end{align*}
so $\Delta_p$ is the usual Laplacian on $\mathbb{R}^d$ acting on the variable $p$ and $\nabla_\xi $ is the gradient on $\mathbb{R}^d$ acting on the variable $\xi$. The operator $L$ is hypoelliptic and generates the Markov process $\{X_{t}\}_{t\ge0}$. We let $P_t$ denote the semigroup generated by $L$ and
 $\Gamma\left( f \right)=\frac{1}{2}Lf^2 -fLf $ denote the \emph{carr\'e du champ} operator for $L$. One may show that
\[
\Gamma(f)= \frac{1}{2} \| \nabla_p f \|^2.
\]
The following is \cite[Proposition 2.8]{BaudoinGordinaMariano2020} with $\sigma=1$.

\begin{proposition}[Reverse log Sobolev inequality]\label{p.rlsi}
Let $f \in C^1(\mathbb{R}^d \times \mathbb{R}^{d})$ be a nonnegative bounded function. One has
\begin{multline*}
 \sum_{i=1}^d \left( \frac{\partial \ln P_t f}{\partial p_i } -\frac{1}{2} t \frac{\partial \ln P_t f}{\partial \xi_i } \right)^2 +\frac{t^2}{12}  \left( \frac{\partial \ln P_t f}{\partial \xi_i }\right)^2
\\
 \leqslant \frac{2}{ t P_t f } (P_t (f \ln f)  - P_tf \ln P_t f ).
\end{multline*}
\end{proposition}

The fact that reverse log Sobolev inequalities imply the Wang-type Harnack inequalities for general Markov operators is by now well-known (see for example \cite[Proposition 3.4]{BaudoinBonnefont2012} and in the infinite-dimensional context \cite[Propositions 2.10 and 4.9]{BaudoinGordinaMelcher2013}). We deduce then the following statement.

\begin{theorem}[Wang-type Harnack inequality]\label{t.wang}
Suppose $f$ be a nonnegative bounded Borel function on $\mathbb{R}^d \times \mathbb{R}^{d}$, then for every $t>0$, $(p,\xi),(p',\xi') \in \mathbb{R}^d \times \mathbb{R}$ and $\alpha\in(1,\infty)$
\begin{align*}
& (P_t f)^\alpha (p,\xi)
\\
& \leqslant (P_t f^\alpha )(p',\xi') \exp \left( \frac{3}{4-\sqrt{13}}  \frac{\alpha}{\alpha -1} \left( \frac{\| p-p'\|^2}{t} +\frac{\| \xi-\xi'\|^2}{t^3} \right) \right).
\end{align*}
\end{theorem}

\begin{proof}
Note first that, for any $a, b\in \R$, we have the elementary inequality
\begin{multline*}
 \left( a- \frac{1}{2}tb \right)^{2}+\frac{t^{2}}{12}b^{2}=a^{2}-atb+\frac{t^{2}b^{2}}{3}
\\
\geqslant a^{2}-\frac{\nu a^{2}}{2}-\frac{t^{2}b^{2}}{2\nu}+\frac{t^{2}b^{2}}{3}=\left( 1-\frac{\nu}{2}\right)a^{2}
+ \left( \frac{1}{3}-\frac{1}{2\nu}\right)t^{2}b^{2}
\end{multline*}
for any $\nu>0$. Setting the coefficients equal we get the minimum at
\[
1-\frac{\nu}{2}= \frac{1}{3}-\frac{1}{2\nu}=\frac{4-\sqrt{13}}{6}.
\]
Thus, we have that
\begin{multline*}
\sum_{i=1}^d \left( \frac{\partial \ln P_t f}{\partial p_i } -\frac{1}{2} t \frac{\partial \ln P_t f}{\partial \xi_i } \right)^2 +\frac{t^2}{12}  \left( \frac{\partial \ln P_t f}{\partial \xi_i }\right)^2 \\
\geqslant \frac{4-\sqrt{13}}{6} \left( \|\nabla_p \ln P_t f \|^2+ t^2 \|\nabla_\xi \ln P_t f \|^2  \right).
 \end{multline*}
Therefore, the reverse log Sobolev inequality in Proposition \ref{p.rlsi} implies that
\[
 \frac{4-\sqrt{13}}{6} \left( \|\nabla_p \ln P_t f \|^2+ t^2 \|\nabla_\xi \ln P_t f \|^2  \right) \leqslant \frac{2}{ t P_t f } (P_t (f \ln f)  -P_tf \ln P_t f ).
 \]
One may now integrate this inequality as in the proof of \cite[Proposition 3.4]{BaudoinBonnefont2012} or \cite[Proposition 2.10]{BaudoinGordinaMelcher2013} to get the desired result.
\end{proof}

It is known that Wang-type Harnack inequalities as in Theorem \ref{t.wang} are equivalent to certain $L^q$ estimates for the heat kernel of the diffusion. More specifically, for $t>0$, we denote by $p_t(\cdot,\cdot):\mathbb{R}^{2d}\times\mathbb{R}^{2d}\rightarrow\mathbb{R}$ the fundamental solution for the Kolmogorov operator on $\mathbb{R}^{2d}$, and also when appropriate $p_t(\cdot)=p_t(0,\cdot):\mathbb{R}^{2d}\rightarrow \mathbb{R}$. The following result is an immediate consequence of Theorem \ref{t.wang}; see for example  \cite[Lemma 2.11]{BaudoinGordinaMelcher2013} taking $q=1/(\alpha-1)$ for $\alpha\in(1, 2)$ and $y=0$.

\begin{prop}[Integrated Harnack inequality]\label{p.intharnack}
For any $t>0$, $(p,\xi)\in\mathbb{R}^d \times \mathbb{R}^{d}$, and $q\in(1,\infty)$,
\begin{multline*}
\left(\int_{\mathbb{R}^d\times\mathbb{R}^d} \left[\frac{p_t((p,\xi),(p',\xi'))}{p_t(p',\xi')}\right]^{q} p_t(p',\xi')\,dp'\,d\xi' \right)^{1/q} \\
	\leqslant \exp \left( \frac{3(1+q)}{4-\sqrt{13}}   \left( \frac{\| p\|^2}{t} +\frac{\| \xi\|^2}{t^3} \right) \right)
\end{multline*}
\end{prop}

\subsection{Generalized Kolmogorov diffusions}

We now consider the non-Gaussian diffusions
\[
Y_t=\left(B_t , \int_0^t F( B_s ) ds \right),
\]
where $F: \R^d\rightarrow\mathbb{R}^r$ is a $C^1$-function.
Similar diffusions were considered in \cite{BaudoinGordinaMariano2020}. These processes have infinitesimal generator
\[
L=\frac{1}{2}\Delta_p+F(p) \cdot\frac{\partial}{\partial \xi}.
\]

\subsubsection{$F:\R^d\to\R$}
We initially restrict our attention to the case $r=1$. In this case, for each $\alpha,\beta \geqslant 0$, we define a first-order differential symmetric bilinear form $\Gamma^{\alpha,\beta}: C^{\infty}\left( \mathbb{R}^d \times \mathbb{R} \right) \times C^{\infty}\left( \mathbb{R}^d \times \mathbb{R} \right) \rightarrow \mathbb{R}$ associated to the diffusion operator $L$ by
\begin{equation*}
	\Gamma^{\alpha,\beta} (f,g):=\sum_{i=1}^d \left( \frac{\partial f}{\partial p_i } -\alpha \frac{\partial f}{\partial \xi } \right)\left( \frac{\partial g}{\partial p_i } -\alpha \frac{\partial g}{\partial \xi } \right) +\beta  \left( \frac{\partial f}{\partial \xi }\right)\left( \frac{\partial g}{\partial \xi }\right).
\end{equation*}
In particular,
\begin{equation*}
	\Gamma^{\alpha,\beta} (f):=\Gamma^{\alpha,\beta} (f,f)=\sum_{i=1}^d \left( \frac{\partial f}{\partial p_i } -\alpha \frac{\partial f}{\partial \xi } \right)^2 +\beta \left( \frac{\partial f}{\partial \xi }\right)^2.
\end{equation*}
Note that the \emph{carr\'e du champ} operator of $L$ is given by
\[
\Gamma(f)=\sum_{i=1}^d \left( \frac{\partial f}{\partial p_i }  \right)^2=\Gamma^{0,0} (f).
\]
We also consider
\begin{equation*}
\Gamma_2^{\alpha,\beta} (f):= \frac{1}{2} L \Gamma^{\alpha,\beta} (f) -\Gamma^{\alpha,\beta} (f,Lf).
\end{equation*}
Recall that the \emph{intrinsic (control) distance} associated to $\Gamma^{\alpha,\beta}$ is defined as
\begin{multline*}
 d_{\alpha,\beta} ( (p,\xi), (p^{\prime},\xi^{\prime}))
\\
: =\sup \left\{ f(p,\xi) - f(p^{\prime},\xi^{\prime}): \, f \in C^{\infty}\left( \mathbb{R}^d \times \mathbb{R} \right), \,   \Gamma^{\alpha,\beta} (f) \leqslant  1\right\}.
\end{multline*}
The control distance may be computed explicitly as follows.

\begin{lemma}\label{l.2.4} For $\alpha\geqslant 0$ and $\beta>0$,
\[
d_{\alpha,\beta} ( (p,\xi), (p',\xi'))^2=\frac{1}{\beta} \left( \alpha \sum_{i=1}^d (p'_i-p_i) +\xi'-\xi  \right)^2+\sum_{i=1}^d (p_i'-p_i)^2.
\]
\end{lemma}
\begin{proof}
Let $ f \in C^{\infty}\left( \mathbb{R}^d \times \mathbb{R} \right)$. Consider the function $ g \in C^{\infty}\left( \mathbb{R}^d \times \mathbb{R} \right)$ such that
\[
g\left( p, \frac{1}{\sqrt{\beta}} \left( \alpha \sum_{i=1}^d p_i +\xi \right) \right)=f( p,\xi).
\]
We have then
\[
\Gamma^{\alpha,\beta} (f)= \sum_{i=1}^d \left( \frac{\partial g}{\partial p_i }\right)^2 +\left( \frac{\partial g}{\partial \xi }\right)^2.
\]
Therefore
\begin{align*}
	d_{\alpha,\beta} &( (p,\xi), (p',\xi'))  \\
	&=\sup  \left\{  g\left( p, \frac{1}{\sqrt{\beta}} \left( \alpha \sum_{i=1}^d p_i +\xi \right) \right) - g\left( p', \frac{1}{\sqrt{\beta}} \left( \alpha \sum_{i=1}^d p_i' +\xi' \right) \right) : \right. \\
& \qquad\qquad\left. \, g \in C^{\infty}\left( \mathbb{R}^d \times \mathbb{R} \right), \,    \sum_{i=1}^d \left( \frac{\partial g}{\partial p_i }\right)^2 +\left( \frac{\partial g}{\partial \xi }\right)^2 \leqslant 1\right\}
\end{align*}
and the conclusion follows from the standard Euclidean case.
\end{proof}

For now, we make the following assumption on the function $F$ defining the diffusion.

\begin{assumption}\label{A} For $F:\mathbb{R}^d\to\mathbb{R}$, there exist constants $m,M >0$ such that for every $i=1,\cdots,d$ and $p\in \mathbb{R}^d$
\begin{align*}
m \leqslant \frac{\partial F}{\partial p_i} (p) \leqslant  M.
\end{align*}
\end{assumption}

Note that the lower bound on the derivatives implies for example that the range of $F$ is unbounded. Under this assumption, we have the following key proposition.

\begin{proposition}\label{p.key}
Suppose that $F:\mathbb{R}^d\to\mathbb{R}$ satisfies Assumption \ref{A}. Then for every $f \in C^{\infty}\left( \mathbb{R}^d \times \mathbb{R} \right)$,
\[
\Gamma_2^{\alpha,\beta} (f) \geqslant -\frac{M-m}{4\alpha} \Gamma(f) + m \sum_{i=1}^d \left( \alpha \left(\frac{\partial f}{\partial \xi }\right)^2 -\frac{\partial f}{\partial \xi } \frac{\partial f}{\partial p_i }\right).
\]
\end{proposition}

\begin{proof}
Direct computation shows that
\begin{align*}
& \Gamma_2^{\alpha,\beta}(f)=\sum_{i,j=1}^d \left( \frac{\partial^2 f}{\partial p_i \partial p_j} \right)^2-2\alpha \sum_{i,j=1}^d  \frac{\partial^2 f}{\partial p_i \partial p_j} \frac{\partial^2 f}{\partial p_i \partial \xi}
\\
&
+(d \alpha^2+\beta)\sum_{i=1}^d \left( \frac{\partial^2 f}{\partial p_i \partial \xi} \right)^2
-\langle \nabla_p f , \nabla_p F \rangle \frac{\partial f}{\partial \xi}+\alpha \left( \frac{\partial f}{\partial \xi }\right)^2 \left( \sum_{i=1}^d \frac{\partial F}{\partial p_i} \right).
\end{align*}
Completing the squares and using $\beta \geqslant 0$ implies that
\[
\sum_{i,j=1}^d \left( \frac{\partial^2 f}{\partial p_i \partial p_j} \right)^2-2\alpha \sum_{i,j=1}^d  \frac{\partial^2 f}{\partial p_i \partial p_j} \frac{\partial^2 f}{\partial p_i \partial \xi} +(d \alpha^2+\beta)\sum_{i=1}^d \left( \frac{\partial^2 f}{\partial p_i \partial \xi} \right)^2 \geqslant 0.
\]
Therefore
\begin{align*}
\Gamma_2^{\alpha,\beta}(f) \geqslant -\langle \nabla_p f , \nabla_p F \rangle \frac{\partial f}{\partial \xi}+\alpha \left( \frac{\partial f}{\partial \xi }\right)^2 \left( \sum_{i=1}^d \frac{\partial F}{\partial p_i} \right).
\end{align*}
We then observe that for each $i=1, ..., d$
\begin{align*}
\frac{\partial F}{\partial p_i} \left( \alpha \left( \frac{\partial f}{\partial \xi }\right)^2  -\frac{\partial f}{\partial \xi } \frac{\partial f}{\partial p_i }\right)&= \frac{\partial F}{\partial p_i} \left( \left( \sqrt{\alpha} \frac{\partial f}{\partial \xi } -\frac{1}{2\sqrt{\alpha}} \frac{\partial f}{\partial p_i } \right)^2-\frac{1}{4\alpha} \left( \frac{\partial f}{\partial p_i }\right)^2 \right) \\
 & \geqslant  m \left( \sqrt{\alpha} \frac{\partial f}{\partial \xi } -\frac{1}{2\sqrt{\alpha}} \frac{\partial f}{\partial p_i } \right)^2-\frac{M}{4\alpha} \left( \frac{\partial f}{\partial p_i }\right)^2 \\
 &= m \left( \alpha \left( \frac{\partial f}{\partial \xi }\right)^2  -\frac{\partial f}{\partial \xi } \frac{\partial f}{\partial p_i }\right)- \frac{M-m}{4\alpha} \left( \frac{\partial f}{\partial p_i }\right)^2
\end{align*}
and the desired result follows from summing over $i$.
\end{proof}

We will use Proposition \ref{p.key} and the fact that  $\Gamma^{\alpha, \beta}$ and $\Gamma$ satisfy the commutation relation
\[
\Gamma^{\alpha, \beta}(f, \Gamma(f))= \Gamma( f, \Gamma^{\alpha,\beta}(f))
\]
to prove a reverse log Sobolev inequality for the generalized Kolmogorov operator. This commutation relation is similar to \cite[Hypothesis 1.2]{BaudoinGarofalo2017}.

\begin{proposition}[Reverse log Sobolev inequality]\label{reverse log-sob}
	Suppose that $F:\mathbb{R}^d\to\mathbb{R}$ satisfies Assumption \ref{A}. Then for all nonnegative bounded $f \in C^1(\mathbb{R}^d \times \mathbb{R})$ and $t>0$
	\begin{multline*}
 \sum_{i=1}^d \left( \frac{\partial \ln P_t f}{\partial p_i } -\frac{m}{2} t \frac{\partial \ln P_t f}{\partial \xi } \right)^2 +\frac{m^2}{12} t^2  \left( \frac{\partial \ln P_t f}{\partial \xi }\right)^2 \\ \leqslant  \frac{M}{ m t P_t f } (P_t (f \ln f)  -P_tf \ln P_t f ).
	\end{multline*}
\end{proposition}

\begin{proof}
We can assume that $f$ is smooth and rapidly decreasing.
Let $t>0$ and for $\alpha\left(  \cdot \right), \beta\left( \cdot \right) \in C^{1}\left( [0, t]\right)$, consider the functional
\[
\phi\left( s \right):=P_s ( (P_{t-s} f) \Gamma^{\alpha (s), \beta (s) } (\ln P_{t-s} f)), \quad 0 \leqslant  s \leqslant  t.
\]
Straightforward computation by the chain rule shows that
\begin{align*}
 \phi^{\prime}\left( s \right)=&2P_s ((P_{t-s} f)  \Gamma_2^{\alpha (s), \beta (s) } (\ln P_{t-s} f))
\\
& -2\alpha^{\prime}(s)\sum_{i=1}^d P_s \left( (P_{t-s} f) \frac{\partial \ln P_{t-s} f}{\partial \xi } \frac{\partial \ln P_{t-s}f}{\partial p_i } \right) \\
& +2\alpha(s)\alpha^{\prime}(s)P_s \left( (P_{t-s} f)  \left(  \frac{\partial \ln P_{t-s} f}{\partial \xi} \right)^2\right) \\
&+\beta^{\prime}(s) P_s \left( (P_{t-s} f)  \left(  \frac{\partial \ln P_{t-s} f}{\partial \xi} \right)^2\right).
\end{align*}
In particular, taking $\alpha(s)=\frac{m}{2} (t-s)$ and $\beta (s)=\frac{m^2}{12} (t-s)^2$ and
\[
	\varphi (s)=(t-s)\phi(s), \quad 0 \leqslant  s \leqslant  t,
\]
we have
\begin{align*}
	\varphi^{\prime}(s)  &=- \phi(s) + (t-s)\phi'(s)
\\
	&\geqslant  -\frac{M}{m} P_s( (P_{t-s} f)  \| \nabla_p \ln P_{t-s} f \|^2) =  -\frac{M}{m} P_s( (P_{t-s} f)  \Gamma(\ln P_{t-s} f)),
\end{align*}
by Proposition \ref{p.key}. Therefore,
\[
\varphi (0) \leqslant  \frac{M}{m} \int_0^t P_s( (P_{t-s} f) \Gamma(\ln P_{t-s} f)) ds.
\]
We now observe that
\[
 \int_0^t P_s( (P_{t-s} f) \Gamma(\ln P_{t-s} f)) ds= P_t(f \ln f)  -P_tf \ln P_t f,
\]
and thus
	\[
t (P_t f )  \Gamma^{\alpha (0), \beta (0) } (\ln P_{t} f) \leqslant  \frac{M}{m} (P_t(f \ln f)  -P_tf \ln P_t f)
\]
which may be re-written to give the desired result.
\end{proof}

As in Section \ref{s.fk}, a Wang-type Harnack inequality now follows from the reverse log Sobolev inequality.

\begin{theorem}[Wang-type Harnack inequality]\label{t.whi}
	Suppose that $F:\mathbb{R}^d\to\mathbb{R}$ satisfies Assumption \ref{A}. Then for every nonnegative Borel bounded $f$ on $\mathbb{R}^d \times \mathbb{R}$, $t>0$, $(p,\xi),(p',\xi') \in \mathbb{R}^d \times \mathbb{R}$, and $\alpha >1$, we have
\begin{align*}
 & (P_t f)^\alpha \left( p,\xi \right) \leqslant  C_{\alpha}\left( t, \left( p,\xi \right), \left( p^{\prime},\xi^{\prime} \right)\right)(P_t f^\alpha )(p^{\prime},\xi^{\prime}),
\end{align*}
where
\begin{multline*}
C_{\alpha}\left( t, \left( p,\xi \right), \left( p^{\prime},\xi^{\prime} \right)\right)
\\
:=\exp \left(  \frac{\alpha M}{4m(\alpha-1)t} \left(  \frac{12}{m^2t^2}  \left( \frac{mt}{2} \sum_{i=1}^d(p_i-p_i') +(\xi-\xi')\right)^2+\|p-p'\|^2\right)  \right).
\end{multline*}
\end{theorem}

\begin{proof}
As before we assume that $f$ is rapidly decreasing. Let $t>0$ be fixed and $(p,\xi),(p',\xi') \in \mathbb{R}^d \times \mathbb{R}^d$. We observe first that the reverse log Sobolev inequality in Proposition \ref{reverse log-sob} can be rewritten
\[
\Gamma^{\frac{m}{2}t, \frac{m^2}{12}t^2} (\ln P_t f)\leqslant   \frac{M}{ tm P_t f } (P_t (f \ln f)  -P_tf \ln P_t f ).
\]
We can now integrate the previous inequality as in the proof of \cite[Proposition 3.4]{BaudoinBonnefont2012} or \cite[Proposition 2.10]{BaudoinGordinaMelcher2013} and deduce
\[
(P_t f)^\alpha (p,\xi) \leqslant  (P_t f^\alpha )(p',\xi') \exp \left( \frac{\alpha M}{4m(\alpha-1)} \frac{d_t^2((p,\xi),(p',\xi'))}{t}  \right).
\]
	where $d_t$ is the control distance associated to the gradient $\Gamma^{\frac{mt}{2}, \frac{m^2t^2}{12}}$. The result now follows by Lemma \ref{l.2.4} with the chosen $\alpha$ and $\beta$.
\end{proof}

Now we modify Assumption \ref{A} as follows.

\begin{assumption}\label{A2} For $F:\mathbb{R}^d\to\mathbb{R}$, there exist constants $m,M >0$ and a non-empty subset $I \subset \{ 1, \cdots, d\}$ such that for every $i \in I$ and $p\in \mathbb{R}^d$
\begin{align*}
m \leqslant \frac{\partial F}{\partial p_i} (p) \leqslant  M
\end{align*}
and for every $i \notin I$ and $p\in \mathbb{R}^d$,  $\frac{\partial F}{\partial p_i} (p) =0$.
\end{assumption}

\begin{notation}
	For a subset $J\subset \{ 1, \cdots, d\}$ and $p\in\R^d$, let $p_J:= (p_i)_{i\in J}\in \R^J$ and
	\[ \|p\|_J^2:= \|p_J\|_{\mathbb{R}^J}^2 = \sum_{i\in J} p_i^2. \]
\end{notation}

\begin{theorem}[Wang-type Harnack inequality]\label{t.whi2}
	Suppose that $F:\mathbb{R}^d\to\mathbb{R}$ satisfies Assumption \ref{A2}. Then for every nonnegative Borel bounded $f$ on $\mathbb{R}^d \times \mathbb{R}$, $t>0$, $(p,\xi),(p',\xi') \in \mathbb{R}^d \times \mathbb{R}$, and $\alpha >1$, we have
\begin{align*}
 & (P_t f)^\alpha \left( p,\xi \right) \leqslant  C_{\alpha,I}\left( t, \left( p,\xi \right), \left( p^{\prime},\xi^{\prime} \right)\right)(P_t f^\alpha )(p^{\prime},\xi^{\prime}),
\end{align*}
with
\begin{align*}
 &C_{\alpha,I}\left( t, \left( p,\xi \right), \left( p^{\prime},\xi^{\prime} \right)\right)
\\
	&:= \exp \left(  \frac{\alpha M}{4m(\alpha-1)t} \left(  \frac{12}{m^2t^2}  \left( \frac{mt}{2} \sum_{i \in I} (p_i-p_i') +(\xi-\xi')\right)^2+\|p-p'\|_I^2\right)  \right)
\\
 &\qquad\times \exp \left(  \frac{\alpha }{4(\alpha-1)t} \|p-p'\|_{I^c}^2 \right).
\end{align*}
\end{theorem}

\begin{proof}
We  consider the diffusion
\[
Y^I_t=\left(B^I_t , \int_0^t F( B_s ) ds \right),
\]
where $B^I_t=(B^i_t)_{i \in I, t \geqslant 0}$. The infinitesimal generator for this process is
\[
L^I=\sum_{i \in I}\frac{\partial^2}{\partial p^2_i} +F(p) \frac{\partial}{\partial \xi}.
\]
	We note that for $i \notin I$, $\frac{\partial F}{\partial p_i}=0$, thus there exists a function $F^I:\R^I\to \R$ such that for every $p \in \mathbb{R}^d$, $F(p)=F^I (p_I)$. This function $F^I$ satisfies Assumption \ref{A} on $\R^I$. Therefore by Theorem \ref{t.whi}  the semigroup $P_t^I$ of the diffusion $(Y_t^I)_{t \geqslant 0}$ satisfies for any bounded Borel function $f:\mathbb{R}^I \times \mathbb{R}\to \mathbb{R}$ the Wang-type Harnack inequality
\begin{align*}
 & (P^I_t f)^\alpha \left( p_I,\xi \right) \leqslant  A^I_{\alpha}\left( t, \left( p_I,\xi \right), \left( p_I^{\prime},\xi^{\prime} \right)\right)(P^I_t f^\alpha )(p_I^{\prime},\xi^{\prime}),
\end{align*}
where
\begin{multline*}
A^I_{\alpha}\left( t, \left( p_I,\xi \right), \left( p_I^{\prime},\xi^{\prime} \right)\right)
\\
:=\exp \left(  \frac{\alpha M}{4m(\alpha-1)t} \left(  \frac{12}{m^2t^2}  \left( \frac{mt}{2} \sum_{i\in I} (p_i-p_i') +(\xi-\xi')\right)^2+\|p-p'\|_I^2\right)  \right).
\end{multline*}
On the other hand, the infinitesimal generator of the diffusion
\[
Y_t=\left(B_t , \int_0^t F( B_s ) ds \right),
\]
can be written as
\[
L=\sum_{i \notin I} \frac{\partial^2}{\partial p^2_i} +L^I
\]
and the associated semigroup $P_t$ satisfies for any nonnegative bounded Borel functions $f: \mathbb{R}^I \times \mathbb{R}\to \mathbb{R}$ and $g: \mathbb{R}^{I^c} \to \mathbb{R}$
 \[
 P_t ( f \otimes g)(p ,\xi)= (P_t^I f )(p_I, \xi) (Q^{I^c}_t g) ( p_{I^c})
 \]
 where $ f \otimes g$ is the function defined by $( f \otimes g)(p ,\xi)= f(p_I, \xi) g(p_{I^c})$ and where $Q^{I^c}_t$ is the Gaussian heat semigroup in $\mathbb{R}^{I^c}$ generated by $\sum_{i \notin I} \frac{\partial^2}{\partial p^2_i}$. This Gaussian semigroup $Q^{I^c}_t$ satisfies the Wang-type Harnack inequality

 \begin{align*}
 & (Q^{I^c}_t g)^\alpha \left( p_{I^c} \right) \leqslant \exp \left(  \frac{\alpha }{4(\alpha-1)t} \|p-p'\|_{I^c}^2 \right) (Q^{I^c}_t g^\alpha )(p_{I^c}^{\prime}).
 \end{align*}

 We therefore deduce that
 \begin{align*}
 & (P_t (f \otimes g))^\alpha \left( p,\xi \right) \leqslant  C_{\alpha,I}\left( t, \left( p,\xi \right), \left( p^{\prime},\xi^{\prime} \right)\right)(P_t (f\otimes g)^\alpha )(p^{\prime},\xi^{\prime}),
\end{align*}
with
\[
C_{\alpha,I}\left( t, \left( p,\xi \right), \left( p^{\prime},\xi^{\prime} \right)\right)=A^I_{\alpha}\left( t, \left( p_I,\xi \right), \left( p_I^{\prime},\xi^{\prime} \right)\right) \exp \left(  \frac{\alpha }{4(\alpha-1)t} \|p-p'\|_{I^c}^2 \right) .
\]
Since this holds for any nonnegative bounded Borel functions $f: \mathbb{R}^I \times \mathbb{R}\to \mathbb{R}$ and $g: \mathbb{R}^{I^c} \to \mathbb{R}$, we conclude that the Wang-type Harnack inequality for $P_t$ holds for any nonnegative bounded Borel function $\mathbb{R}^d \times \mathbb{R} \to \mathbb{R}$.
\end{proof}

Now the same kind of tensorization argument from the proof of Theorem \ref{t.whi2} can be adapted to deal with certain functions $F$ taking values in $\R^r$. In particular, we make the following assumption.

\begin{assumption}\label{A3}
	For $F:\mathbb{R}^d\to\mathbb{R}^r$, there exist non-empty disjoint subsets $I_1,\ldots,I_r \subset \{ 1, \cdots, d\}$ and constants $m_1,M_1,\ldots,m_r,M_r >0$ such that for each $j=1,\ldots,r$, for every $i \in I_j$ and $p\in \mathbb{R}^d$,
\begin{align*}
m_j \leqslant \frac{\partial F_j}{\partial p_i} (p) \leqslant  M_j,
\end{align*}
and, for every $i \notin I_j$ and $p\in \mathbb{R}^d$,  $\frac{\partial F_j}{\partial p_i} (p) =0$.
\end{assumption}

Note that the assumption that the subsets $I_j$ are non-empty excludes $F$ having some constant coordinates, and again the lower bound on the derivatives implies more generally that the range of $F$ is unbounded in each coordinate. The disjointness of the $I_j$'s implies that the coordinates of $F(B)$ are independent.

\begin{theorem}[Wang-type Harnack inequality]\label{t.whi3}
	Suppose $F:\mathbb{R}^d\to\mathbb{R}^r$ satisfies Assumption \ref{A3}. Then for every nonnegative Borel bounded $f$ on $\mathbb{R}^d \times \mathbb{R}^r$, $t>0$, $(p,\xi),(p',\xi') \in \mathbb{R}^d \times \mathbb{R}^r$, and $\alpha >1$, we have
\begin{multline*}
(P_t f)^\alpha \left( p,\xi \right) \leqslant  \left(\prod_{j=1}^r A^j_{\alpha}\left( t, \left( p,\xi \right), \left( p^{\prime},\xi^{\prime} \right)\right) \right)\\ \times\exp \left(  \frac{\alpha }{4(\alpha-1)t} \|p-p'\|_{I^c}^2 \right)(P_t f^\alpha )(p^{\prime},\xi^{\prime}),
\end{multline*}
where $I^c:= (\cup_{j=1}^r I_j)^c$ and
\begin{multline*}
A^j_{\alpha}\left( t, \left( p,\xi \right), \left( p^{\prime},\xi^{\prime} \right)\right)
	:= A^j_{\alpha}\left( t, \left( p_{I_j},\xi_j \right), \left( p_{I_j}^{\prime},\xi_j^{\prime} \right)\right)\\
	:=\exp \left(  \frac{3\alpha M_j}{m_j^3(\alpha-1)t^3} \left( \frac{m_jt}{2}  \sum_{i\in I_j} (p_i-p_i') +(\xi_j-\xi_j')\right)^2\right) \\
	\times\exp\left(\frac{\alpha M_j}{4m_j(\alpha-1)t} \|p-p'\|_{I_j}^2  \right).
\end{multline*}
\end{theorem}

We omit the proof here, but the argument follows exactly as in the proof of Theorem \ref{t.whi2}, relying on the key assumption that the coordinates of $F$ rely on \textit{disjoint} collections of coordinates, and thus the generator may be written as
\[ L= \sum_{j=1}^r L^{I_j} + \sum_{i\notin\cup I_j} \frac{\partial^2}{\partial p_i^2}
= \sum_{j=1}^r \left(\sum_{i\in I_j} \frac{\partial^2}{\partial p_i^2} + F_j(p)\frac{\partial}{\partial \xi_j}\right) + \sum_{i\notin\cup I_j} \frac{\partial^2}{\partial p_i^2}. \]

As in Section \ref{s.fk}, we know this Wang-type Harnack inequality is equivalent to an integrated Harnack inequality for the heat kernel of the diffusion, which we express in the following Proposition. For $t>0$, let $p_t:\mathbb{R}^{d+r}\times\mathbb{R}^{d+r}\rightarrow\mathbb{R}$ denote the fundamental solution, that is, the heat kernel, of the generator on $\mathbb{R}^{d+r}$.

\begin{prop}[Integrated Harnack inequality]
\label{p.gintharnack3}
	Suppose $F:\mathbb{R}^d\to\mathbb{R}^r$ satisfies Assumption \ref{A3}.Then
	for any $t>0$, $(p,\xi)\in\mathbb{R}^d \times \mathbb{R}^r$, and $q\in(1,\infty)$,
\begin{multline*}
\left(\int_{\mathbb{R}^d\times\mathbb{R}^r} \left[\frac{p_t((p,\xi),(p',\xi'))}{p_t(p',\xi')}\right]^{q} p_t(p',\xi')\,dp'\,d\xi' \right)^{1/q} \\
	\leqslant \left(\prod_{j=1}^r A_{j,q}(p,\xi)\right) \exp \left(  \frac{1+q}{4t} \|p\|_{I^c}^2  \right)
\end{multline*}
	where $I^c:= (\cup_{j=1}^r I_j)^c$ and
		\begin{multline*}
			A_{j,q}(p,\xi) := A_{j,q}(p_{I_j},\xi_j) \\:= \exp \left(  \frac{3(1+q)M_j}{m_j^3t^3}   \left( \frac{m_jt}{2}  \sum_{i\in I_j}p_i +\xi_j \right)^2  \right)
	\exp \left(  \frac{(1+q)M_j}{4m_jt} \|p\|_{I_j}^2  \right).
\end{multline*}
\end{prop}

With this result in hand, we are now ready to introduce the infinite-dimensional diffusions of interest.

\section{Infinite-dimensional results}
\label{s.infinite}

Let $(W,H,\mu)$ be an abstract Wiener space, where $W$ is a separable Banach space equipped with Gaussian measure $\mu$ and $H$ is the associated Cameron-Martin Hilbert space. For background about abstract Wiener spaces, see for example \cite{BogachevGaussianMeasures} or \cite{KuoLNM1975}. We use $\|\cdot\|_H$ and $\|\cdot\|_W$ for the norms on $H$ and $W$, respectively. The inner product on $H$ is denoted by $\langle \cdot, \cdot \rangle_H$.

As usual, $H^{\ast}$ denotes the dual space of $H$ which is of course isomorphic to $H$. Furthermore, let $H_\ast$ be the set of $h\in H$ such that  $\left\langle \cdot, h\right\rangle _{H}\in H^{\ast}$ extends to a continuous linear functional on $W$.  We will continue to denote the continuous extension of $\left\langle \cdot, h\right\rangle _{H}$ to $W$ by $\left\langle \cdot, h\right\rangle_{H}$. Equivalently, we can define $H_\ast$ as follows. Let $i:H\rightarrow W$ be the inclusion map, and let $i^*:W^*\rightarrow H^*$ be its transpose so that $i^*\ell:=\ell\circ i$ for all $\ell\in W^*$.  Then
\[
H_{\ast} = \left\{ h\in H: \langle\cdot, h\rangle_H\in \operatorname{Range}(i^{\ast})\subset H^{\ast} \right\}.
\]
Note that since $H$ is a dense subspace of $W$, $i^{\ast}$ is injective and thus has a dense range.  For any $h \in H$ the map $h\mapsto\langle\cdot, h\rangle_H \in H^{\ast}$ is a linear isometric isomorphism, therefore it follows that if $h \in H_{\ast}$ then $h \mapsto \langle\cdot, h\rangle_H\in W^{\ast}$ is a linear isomorphism also, and so $H_{\ast}$ is a dense subspace of $H$.

Suppose  that $P:H\rightarrow H$ is a finite rank orthogonal projection
such that $PH\subset H_{\ast}$. Let $\left\{  e_{j}\right\}_{j=1}^{n}$ be an
orthonormal basis for $PH$ and $\ell_{j}:=\left\langle \cdot, e_{j}\right\rangle_{H}\in W^{\ast}$. Then we may extend $P$ to a unique continuous operator from $W$ $\rightarrow H$ (still denoted by $P$) by letting
\begin{equation}\label{e.proj}
Pw:=\sum_{j=1}^{n}\left\langle w, e_{j}\right\rangle _{H}e_{j}=\sum_{j=1}^{n}\ell_{j}\left(  w\right)  e_{j}\text{ for all } w\in W.
\end{equation}
For more details on these projections see \cite{DriverGordina2008}.

\begin{notation}
\label{n.proj}
Let $\mathrm{Proj}(W)$ denote the collection of finite rank projections
on $W$ such that
\begin{enumerate}
\item $PW\subset H_*$ and
\item $P|_H:H\rightarrow H$ is an orthogonal projection, that is, $P$ has the form given in equation \eqref{e.proj}.
\end{enumerate}
\end{notation}

We will say a function $f:W \rightarrow\mathbb{R}$ is a (smooth) cylinder function if it may be written as $f=\phi\circ P$ for some $P\in\operatorname*{Proj}\left(  W\right)$ and
some (smooth) function $\phi:\mathbb{R}^{n} \rightarrow \mathbb{R}$, where $n$ is the rank of $P$.

Given $f:W \rightarrow\mathbb{R}$, we say that $f$ is $H$-differentiable at $w\in W$ if $\varphi(h):= f(w+h)$, considered as a function on $H$, is Fr\'{e}chet differentiable at $0$. If $f$ is $H$-differentiable on $W$, we denote by $\nabla f:W\to H$ the mapping defined by
\[ \langle\nabla f(w),h\rangle_H=\partial_h f(w)= \frac{d}{d\varepsilon}\bigg|_0f(w+\varepsilon h). \]
In particular,  let $\{ e_{k} \}_{k=1}^{\infty}$ be an orthonormal basis of $H$ such that $e_{k} \in H_{\ast}$ for all $k$ and define $\ell_k(w)=\langle w,e_k\rangle_H$. For each $n$, let $H_n$ be the span of $\{e_1, \ldots, e_n\}$ identified with $\mathbb{R}^{n}$, and define $P_{n}\in \operatorname*{Proj}\left(  W\right)$
by
\[
P_{n} : W \rightarrow H_{n}  \subset H_{\ast} \subset H
\]
as in \eqref{e.proj}. Then for a cylinder function of the form $f(w)=\phi\circ P_n(w)=\phi(\ell_1(w),\dots,\ell_n(w))$
\[ \nabla f(w) =\sum_{i=1}^n (\partial_{e_i}\phi)(\ell_1(w),\dots,\ell_n(w))e_i. \]
Similarly we can define the second $H$-derivative $\nabla^{2}$, and finally
\[
\Delta f\left( x \right):=\operatorname{tr} \nabla^{2} f \left( x \right)
\]
whenever $\nabla^{2} f \left( x \right)$ exists and is of trace class.

Let $\mathcal{B}(W)$ be the Borel $\sigma$-algebra on $W$. For $t \geqslant 0$, let $\mu_{t}$ be the rescaled measure $\mu_{t}\left( A \right): = \mu\left( A/\sqrt{t} \right)$ with $\mu_{0}=\delta_{0}$. As was first noted by  L.~Gross in \cite[p. 135]{Gross1967a}, there  exists a stochastic process $\{B_t\}_{t\geq 0}$ with values in $W$ which is a.s.~continuous
in $t$ with respect to the norm topology on $W$, has independent increments, and, for $s<t$,
 $\mathrm{Law}(B_{t}-B_{s})=\mu_{t-s}$, with $B_{0}=0$ a.s. $\{B_{t}\}_{t\ge0}$ is called \emph{standard Brownian motion} on $\left( W, \mu \right)$.

For $A \in \mathcal{B}\left( W \right)$, let $ \mu_{t}\left( x, A \right):=\mu_{t}\left( x- A \right)$.
It is well known that $\{\mu_t\}$ forms a family of Markov transition
kernels, and we may thus view $(B_t, \mathbb{P}^x)$ as a strong Markov
process with state space $W$, where $\mathbb{P}^x$ is the law of
$x+B$.

\subsection{Infinite-dimensional Kolmogorov diffusion}
\label{s.infK}

We now consider $W \times W$ as an abstract Wiener space with respect to the product topology on $W\times W$ with product measure $\mu\otimes\mu$. 
A cylinder function $f:W\times W\rightarrow \R$ is defined analogously as before, that is,
\begin{equation}\label{e.cyl}
	f(p,\xi)=\phi(Pp,P\xi)
\end{equation} for some $\phi:PH\times PH\to\R$ and $P\in\mathrm{Proj}(W)$.

Let $\{e_j\}_{j=1}^\infty\subseteq H_*$ be an orthonormal basis for $H$. Then, for any smooth cylinder function $f$ on $W \times W$ of the form (\ref{e.cyl}), we define
\begin{align*}
& \nabla^\xi f \left( p, \xi\right) := \sum_{j=1}^{\infty} (\partial_{(0,e_j)} f)\left( p, \xi \right) e_j
	= \nabla^\xi_{PH} \phi \left( Pp, P\xi\right),
\\
&  \nabla^p f \left( p, \xi\right) := \sum_{j=1}^{\infty} (\partial_{(e_j,0)} f)\left( p, \xi \right) e_j
	= \nabla^p_{PH} \phi \left( Pp, P\xi\right),
\\
&\Delta^p f \left( p, \xi \right) := \sum_{j=1}^\infty (\partial_{(e_j,0)}^2f) \left( p, \xi \right)
	= \Delta^p_{PH} \phi \left( Pp, P\xi \right).
\end{align*}
Note that the sums in these definitions are finite as we only apply these operators to cylinder functions. Similarly, we may define the operators
\begin{align}\label{e.generator}
 \left( Lf \right) \left( p, \xi \right)
	&:= \sum_{j=1}^{\infty} p_{j}\partial_{(0,e_j)} f\left( p, \xi \right)
		+\frac{1}{2} \sum_{j=1}^\infty \partial_{(e_j,0)}^2f \left( p, \xi \right)
\end{align}
and $\Gamma(f):= \frac{1}{2}Lf^2 -fLf$. Note that for $f$ a cylinder function of the form (\ref{e.cyl})
\begin{align*}
 \left( Lf \right) \left( p, \xi \right)
	&= \langle Pp , \nabla^\xi_{PH} \phi \left( Pp, P\xi\right) \rangle+\frac{1}{2} \Delta^p_{PH} \phi \left( Pp, P\xi \right)
\end{align*}
and
\[ \Gamma(f) = \frac{1}{2} \| \nabla^p_{PH} \phi \|_{PH}^2, \]
and thus $L$ and $\Gamma$ are well-defined independent of the choice of basis in the definition.

We now define an infinite-dimensional Kolmogorov diffusion, namely,
\[
X_t := \left(B_t, \int_0^t B_s\,ds\right).
\]
Using an analogous computation to that in  \cite[Appendix]{MarianoPhDThesis2018}, one may verify that the operator $L$ defined by \eqref{e.generator} is the generator of the diffusion $X_{t}$.

\begin{proposition}
\label{p.approx}
	Let $\{P_n\}_{n=1}^\infty\subset \mathrm{Proj}(W)$ such that $P_n\uparrow I_H$, and consider the processes $\{B_n(t)\}_{t\ge0}$ in $P_nH$ and $\{X_n(t)\}_{t\ge0}$ in $P_nH\times P_nH$ defined by
\[ B_n(t) := P_nB(t) \]
and
\[ X_n(t) := \left(B_n(t), \int_0^t B_n(s)\,ds\right). \]
Then
	\[ \lim_{n\rightarrow\infty} \mathbb{E}\left[\max_{0\leqslant t\leqslant T} \|X(t)-X_n(t)\|_{W\times W}^p \right] = 0 \]
for all $p\in[1,\infty)$ and
	\[ \lim_{n\rightarrow\infty} \max_{0\leqslant t\leqslant T} \|X(t)-X_n(t)\|_{W\times W} =0 \text{ a.s.} \]
\end{proposition}

\begin{proof}
It is show in \cite[Proposition 4.6]{DriverGordina2008} that $B_{n}\left(  t\right):=P_{n}B\left(  t\right)  \in P_{n}H \subset H \subset W$ gives a natural approximation to $B_t$.
Namely, there it is proved that
	\begin{equation}\label{e.BMconv} \lim_{n\rightarrow\infty} \max_{0\leqslant s\leqslant T} \|B(s)-B_n(s)\|_{W} =0 \text{ a.s.}
	\end{equation}
	and
		\begin{equation}\label{e.Lpconv} \lim_{n\rightarrow\infty} \mathbb{E}\left[ \max_{0\leqslant s\leqslant T} \|B(s)-B_n(s)\|_W^p   \right]=0\end{equation}
for all $p\in[1,\infty)$.
	Both convergences for the integral are then a straightforward consequence, since 
	\begin{align*}
		\max_{0\leqslant t\leqslant T} \left\|\int_0^t B(s)\,ds -\int_0^t B_n(s)\,ds\right\|_{W}^p
		&\leqslant \max_{0\leqslant t\leqslant T} \int_0^t \|B(s)-B_n(s)\|_W^p \,ds   \\
		&=  \int_0^T \|B(s)-B_n(s)\|_W^p \,ds   \\
		&\leqslant  T\cdot\max_{0\leqslant s\leqslant T} \|B(s)-B_n(s)\|_W^p,	\end{align*}
and thus the desired convergence of $X_n$ follows.
\end{proof}

We now state the Cameron-Martin type quasi-invariance result for $\nu_t=\mathrm{Law}(X_t)$. We prove this as an application of the main theorem in \cite{Gordina2017}. However, to see how a direct proof would work, the reader may see the proof of Theorem \ref{t.gkqi3} in the next subsection.

\begin{theorem}
	\label{t.kqi}
	For any fixed $t>0$, the measure $\nu_t=\mathrm{Law}(X_t)$ on $W\times W$ is quasi-invariant under the action by elements of the group $H\times H$ given by
	\[ \Phi_t^{(h,k)}(p,\xi) := (p + h, \xi+k+th). \]
	Moreover, the Radon-Nikodym derivative of $d\nu_t^{h,k}:=(\Phi_t^{(h,k)})_*\nu_t$ with respect to $\nu_t$ satisfies, for any $q\in(1,\infty)$,
	\[ \left\|\frac{d\nu_t^{h,k}}{d\nu_t}\right\|_{L^q(W\times W,\nu_t)}
		\leqslant \exp\left(\frac{3(1+q)}{4-\sqrt{13}}\left(\frac{\|h\|_H^2}{t}+\frac{\|k\|_H^2}{t^3}\right)\right). \]
\end{theorem}

\begin{proof}
For fixed $t>0$, define $\Phi_t:(H\times H)\times(W\times W) \rightarrow W\times W$ by
\[ \Phi_t((h_1,h_2),(x_1,x_2)) := (x_1+h_1,x_2+h_2 + th_1). \]
For each $t$, $\Phi_t$ defines a measurable group action of the abelian group $H\times H$ on $W\times W$.
Thus, in light of the estimates in Section \ref{s.fk}, the proof is a straightforward application of Theorem 3.2 of \cite{Gordina2017}. We fix any sequence of projections $\{P_n\}_{n=1}^\infty\subset \mathrm{Proj}(W)$ such that $P_n\uparrow I_H$ and let $H_n=P_nH$. In the context of \cite{Gordina2017}, $H_n\times H_n$ plays the role of both the finite-dimensional unimodular Lie subgroup of $H\times H$ and the sequence of topological subspaces of $W\times W$ such that $\cup_{n=1}^\infty (H_n\times H_n)$ is dense in $W\times W$.
	Let $\nu^n_t:= \mathrm{Law}(X^n_t)$ on $H_n\times H_n$. Proposition \ref{p.approx} implies that
	\[ \int_{W\times W} f\,d\nu_t = \lim_{n\rightarrow\infty} \int_{H_n\times H_n} (f\circ j_n)\,d\nu^n_t \]
	where $j_n:H_n\times H_n\rightarrow W\times W$ is the continuous injection map. Proposition \ref{p.intharnack} provides the necessary uniform estimates on the Radon-Nikodym derivatives of the ``shifted'' measures in the finite-dimensional approximations (see (3.3) in \cite{Gordina2017}), and this completes the proof.
\end{proof}

\begin{remark}\label{r.CMM}
One can contrast this approach with a proof of quasi-invariance through an application of the Cameron-Martin-Maruyama theorem on path space.

	For $t>0$ and for $h,k\in H$, let $X_t^{h,k}:=\Phi_t^{(h,k)}(X_t)$.
	Let $\mathcal{W}_t$ denote the space of continuous paths $\omega:[0,t]\rightarrow W$ with $\omega(0)=0$, equipped with the Gaussian measure $\mu=\mathrm{Law}(B)$, and let $\mathcal{H}_t$ denote the associated Cameron-Martin subspace of finite-energy paths taking values in $H$.
Let $f$ be a bounded continuous function on $W\times W$. Then for any $h,k\in H$, the Cameron-Martin-Maruyama theorem on $\mathcal{W}_t$ implies that, for any $\gamma\in \mathcal{H}_t$, the translation $B\mapsto B +\gamma$ in $\mathcal{W}_t$ gives
\begin{multline*}
	\mathbb{E}[f(X_t^{h,k})]
	= \mathbb{E}\left[f\left(B_t + h,\int_0^t (B_s +h)\,ds+k \right)\right] \\
	= \mathbb{E}\left[f\left(B_t+\gamma(t)+h,\int_0^t (B_s +\gamma(s)+h)\,ds + k \right)
		J_t^{\gamma}(B)\right],
\end{multline*}
where
\[ J_t^\gamma(w) = \exp\left(\int_0^t \langle \dot{\gamma}(s),dw(s)\rangle
	- \frac{1}{2}\int_0^t \|\dot{\gamma}(s)\|_H^2\,ds\right). \]
We know that
\[ \mathbb{E}\left[\exp\left(q\int_0^t \langle \dot{\gamma}(s),dB_s\rangle\right)
	\right]
	= \exp\left(\frac{q^2}{2}\int_0^t \|\dot{\gamma}(s)\|_H^2\,ds\right)
	= \exp\left(\frac{q^2}{2}\|\gamma\|_{\mathcal{H}_t}^2\right) \]
and so
\begin{align*} \mathbb{E}\left[J_t^\gamma(B)^q\right]
	&=  \mathbb{E}\left[\exp\left(q\int_0^t \langle \dot{\gamma}(s),dB_s\rangle\right)
	\right]
	\exp\left(- \frac{q}{2}\int_0^t \|\dot{\gamma}(s)\|_H^2\,ds\right)\\
	&= \exp\left(\frac{q^2-q}{2}\|\gamma\|_{\mathcal{H}_t}^2\right).
\end{align*}
For example, for the path $\gamma(s) = sa + s^2b$ with
	\[ a = -\frac{4}{t}h - \frac{6}{t^2}k \qquad\text{ and } \qquad
		b = \frac{3}{t^2}h + \frac{6}{t^3}k, \]
we have
\[ \mathbb{E}[f(X_t^{(h,k)})] = \mathbb{E}\left[f\left(B_t,\int_0^t B_s\,ds \right)
	J_t^{\gamma}(B)\right]. \]
Straightforward computations show that
\[ \|\gamma\|_{\mathcal{H}_t}^2
	= \frac{4}{t}\|h\|_H^2 + \frac{12}{t^2}\langle h,k\rangle_H +\frac{12}{t^3}\|k\|^2_H \]	
and thus
\begin{multline*}	
	\left\|\frac{d\nu_t^{h,k}}{d\nu_t}\right\|_{L^q(W\times W,\nu_t)}
	\le \|J_t^\gamma(B)\|_{L^q(\mathcal{W}_t}
	= \mathbb{E}[(J_t^\gamma(B))^q]^{1/q}  \\
	= \exp\left(2(q-1)\left(\frac{\|h\|^2_H}{t} + \frac{3\langle h,k\rangle_H}{t^2} +\frac{3\|k\|^2_H}{t^3}\right)\right).
\end{multline*}
(One may look at the proof of Theorem \ref{t.gkqi3} to see how one arrives at this upper bound.)
\end{remark}

\subsection{Generalized Kolmogorov diffusions}
As before, we let $\{B_t\}_{t\ge0}$ denote Brownian motion on $(W,\mu)$. Let $V$ be a normed vector space, and  consider a continuous function $F:W\to V$. Define
\[ Y_t := \left( B_t, \int_0^t F(B_s)\,ds\right). \]

We first consider natural approximations to the process $Y$. When $F$ takes values in a finite-dimensional vector space, these approximations live in finite dimensions, but that assumption is not necessary for the following convergence result.
\begin{proposition}
	\label{p.gapprox}
 Let $\{P_n\}_{n=1}^\infty\subset \mathrm{Proj}(W)$ be any collection of projections such that $P_n\uparrow I_H$, and take
\[ B_n(t) := P_nB(t) \in P_nH \]
and
\[ Y_n(t) := \left(B_n(t), \int_0^t F(B_n(s))\,ds\right) \]
	on $P_nH\times F(P_nH)$.
Then
\[ \lim_{n\rightarrow\infty} \max_{0\leqslant t\leqslant T} \|Y(t)-Y_n(t)\|_{W\times V} =0 \text{ a.s.} \]
\end{proposition}
\begin{proof}
	We have that $B_n$ and $B$ are a.s.~(uniformly) continuous on $[0,T]$ and, by (\ref{e.BMconv}), a.s.~$B_n\to B$ uniformly on $[0,T]$ in the $W$ norm. Let
\begin{multline*}
\Omega':=\bigcap_{n=1}^\infty \{B_n \text{ continuous}\} \\ \cap \{B \text{ continuous}\} \cap \{B_n\to B \text{ uniformly on }[0,T] \text{ in } W\},
\end{multline*}
	and note that $P(\Omega')=1$. For each $\omega\in\Omega'$ there exists $M(\omega)<\infty$ and $N_1(\omega)$ such that  $\|B_n(\omega)(s)\|_W,\|B(\omega)(s)\|_W\leqslant M(\omega)$ for all $s\in[0,T]$ and $n\geqslant N_1(\omega)$. Since $F$ is continuous, it is uniformly continuous on the closed ball of radius $M(\omega)$, and thus $F(B_n(\omega)(s))\to F(B(\omega)(s))$ for all $s\in[0,T]$. More precisely, given $\varepsilon>0$, there exists $\delta=\delta(\varepsilon,M(\omega))>0$ such that for all $x, y$ in the ball of radius $M(\omega)$ in $W$, $\|x-y\|_W<\delta$ implies that $\|F(x)-F(y)\|_V<\varepsilon$. Now choose $N=N(\delta)$ such that for all $n\geqslant N\vee N_1(\omega)$, $\|B_n(\omega)(s)-B(\omega)(s)\|_W<\delta$ for all $s\in[0,T]$.

Now, $F(B_n(\omega))$ and $F(B(\omega))$ are uniformly continuous functions on $[0,T]$, and again for sufficiently large $n$ we have that $\|F(B_n(\omega))\|_V,\|F(B(\omega))\|_V\leqslant C(\omega,T)$ on $[0,T]$. So, for any $\omega\in\Omega'$, $F(B_n(\omega))\to F(B(\omega))$ in $L^1([0,T], V)$, and thus almost surely
\begin{multline*}
	\max_{0\leqslant t\leqslant T} \left\|\int_0^t F(B_n(s))\,ds - \int_0^t F(B(s))\,ds\right\|_V \\
	\leqslant \int_0^T \left\| F(B_n(s)) -  F(B(s))\right\|_V\,ds \xrightarrow[n \to \infty]{} 0.
\end{multline*}
\end{proof}

Now that we have appropriate approximations in place, we formulate the following analogue of Assumption \ref{A3}.

\begin{assumption}\label{B3}
	Suppose $F=(F_1,\ldots, F_r):W\to \R^r$ such that each $F_j$ is $H$-differentiable, and assume that there exist an orthonormal basis  $\{e_i\}_{i=1}^\infty\subset H_*$ of $H$, non-empty disjoint subsets $I_1, \ldots, I_r \subset \mathbb{N}$, and constants $m_1, M_1, \ldots, m_r, M_r > 0$ such that for each $j=1,\ldots, r$, for all $w\in W$
\[
m_j \leqslant \langle \nabla F_j(w),e_i\rangle  \leqslant  M_j, \quad \text{ for all } i \in I_j
\]
and
\[
\langle \nabla F_j(w),e_i\rangle=0, \quad \text{ for all } i \notin I_j.
\]
\end{assumption}

Fix the following notation.

\begin{notation}\label{n.I}
	Given an orthonormal basis $\{e_i\}_{i=1}^\infty$ of $H$ and $J\subset \mathbb{N}$, for $h\in H$ let
	\[ \|h\|_J := \left(\sum_{i\in J} |\langle h,e_i\rangle|^2\right)^{1/2}. \]
\end{notation}

\begin{theorem}\label{t.gkqi3}
	Suppose that Assumption \ref{B3} holds for $F:W\to\mathbb{R}^r$. Fix $h\in H$ and $k\in\R^r$ and let $\nu_t^{(h,k)}$ denote the law of the process
\[
Y_t^{(h,k)} = \left(h+ B_t, \int_0^t F(B_s+h)\,ds +k\right).
\]
If for each $j=1,\ldots, r$,
\begin{equation}\label{e.cond}
\sum_{i\in I_j} |\langle h,e_i\rangle| < \infty,
\end{equation}
then $\nu_t^{(h,k)}$ is mutually absolutely continuous with $\nu_t:=\nu_t^0$ and, for any $q\in(1,\infty)$,
\begin{align*}
\left\|\frac{d\nu_t^{(h,k)}}{d\nu_t}\right\|_{L^q(W\times \R^r,\nu_t)}
		&\leqslant 			\left(\prod_{j=1}^r A_{j,q}(h,k)\right)\exp\left( \frac{1+q}{4t}\|h\|_{I^c}^2\right)
\end{align*}
	where $I^c := \left(\cup_{i=1}^r I_j\right)^c$ and
	\begin{multline*} A_{j,q}(h,k) \\ : = \exp\left(  \frac{3(1+q)M_j}{m_j^3t^3}  \left( {\frac{m_j t}{2}} \sum_{i\in I_j} \langle h,e_i\rangle + k_j \right)^2\right)
	\exp\left( \frac{(1+q)M_j}{4m_jt}\|h\|_{I_j}^2\right) \end{multline*}
with $\{e_i\}_{i=1}^\infty\subset H_*$ is the orthonormal basis, $I_j\subset\mathbb{N}$, and $m_j$ and $M_j$ are the bounds introduced in Assumption \ref{B3}.
\end{theorem}

\begin{example}
	A natural and important class of examples satisfying Assumption \ref{B3} and the condition \eqref{e.cond} includes functions $F=(F_1,\ldots,F_r)$ such that each $F_j$ is a cylinder function (that is, each $I_j$ is a finite set) with
\[
	F_j(w)= \phi_j ( w_{I_j}),
\]
	where $(w_I) := (\langle w, e_i\rangle)_{i\in I}$ for some $\{e_i\}_{i=1}^\infty\subset H_*$ an orthornormal basis and functions $\phi_j: \mathbb{R}^{I_j} \to \mathbb{R}$ each satisfying Assumption \ref{A}. In this case, the quasi-invariance results of Theorem \ref{t.gkqi3} hold for all $(h,k)\in H\times \R^r$.

In particular, one could have  $r=1$ with $\#I_1=d$ and
\[
F(w)= \phi ( \langle w ,e_1\rangle, \cdots ,  \langle w ,e_d \rangle )
\]
	where $e_1,\cdots,e_d \in H_*$ are orthornormal and the function $\phi : \mathbb{R}^d \to \mathbb{R}$ satisfies Assumption \ref{A}.
In this case, with $I=I_1$, $m=m_1$, and $M=M_1$, we would have that
		\begin{align*}
		\left\|\frac{d\nu_t^{(h,k)}}{d\nu_t}\right\|_{L^q(W\times \R,\nu_t)}
		&\leqslant 			\exp\left(  \frac{3(1+q)M}{m^3t^3}  \left( {\frac{m t}{2}} \sum_{i\in I} \langle h,e_i\rangle + k \right)^2\right)\\ &\qquad\times
	\exp\left( \frac{(1+q)M}{4mt}\|h\|_I^2\right)
		\exp\left( \frac{1+q}{4t}\|h\|_{I^c}^2\right)
\end{align*}

\end{example}

\begin{remark}
	Note that the setting of Theorem \ref{t.gkqi3} does not technically fit the hypotheses of the main theorems in \cite{DriverGordina2009} or \cite{Gordina2017}, as starting the diffusion at $(h,k)$ does not constitute a translation of or other group action on the original diffusion. However, the proof given below is a direct analogy to the proofs in those references. Thus one may see that this approach for proving quasi-invariance works more generally; it seems only necessary that the transformation be measurable (of course), but it need not for example correspond to a group action.
\end{remark}

\begin{proof}	
	Fix $t>0$ and $P_0\in\mathrm{Proj}(W)$.  Let $y=(h,k)\in P_0H\times \R^r$ and
$\{P_n\}_{n=1}^\infty$ be an increasing sequence of projections in $\mathrm{Proj}(W)$ such that
	$P_0H\subset H_n=P_nH$ for all $n$ and $P_n|_H\uparrow I_H$. Let $Y_n(t)$ be as in Proposition \ref{p.gapprox} with $V=\R^r$, and let $\nu^{n}_t=\mathrm{Law}(Y_n(t))$ on $H_n\times\R^r$. Furthermore, let $B_n:=P_nB$,
	\[ Y^y_n(t) := \left(h+ B_n(t), \int_0^t F(B_n(s)+h)\,ds +k\right), \]
and $\nu^{n,y}_t=\mathrm{Law}(Y_n^y(t))$ on $H_n\times\R^r$. By their strict positivity, $\nu^n_t$ and $\nu_t^{n,y}$ are mutually absolutely continuous, and we let $J^{n,y}_t$
	denote the Radon-Nikodym derivative of $\nu^{n,y}_t$ with respect to $\nu^n_t$. Letting $p_t^n(y,\cdot)$ denote the density for $\nu^{n,y}_t$ with respect to Lebesgue measure, we have that
		\[ J^{n,y}_t(x) = \frac{p^n_t(y,x)}{p^n_t(0,x)}.
	\]
	By Proposition \ref{p.gintharnack3}, we have
\begin{align*}
	\|J^{n,y}_t\|_{L^q(H_n\times \R^r,\nu_t^n)}
	&\leqslant \left(\prod_{j=1}^r A_{j,q}(h,k)\right)\exp\left( \frac{1+q}{4t}\|h\|_{I^c}^2\right),
\end{align*}
noting that, since $h\in P_0H\subset P_nH$,
	\[
\langle h, P_ne_i\rangle =  \langle P_nh, e_i\rangle = \langle h, e_i\rangle
\]
implies that the upper bound is independent of $n$.

	By an analogous proof to that of Proposition \ref{p.gapprox} one may show that the $Y_n^y$ approximate the process $Y^y$, and thus for any bounded continuous
$f$ on $W\times \R^r$
\begin{equation}
\label{e.5.7}
	\int_{W\times \R^r} f \,d \nu^y_{t}= \lim_{n\rightarrow\infty}\int_{H_n\times\R^r} f\circ i_n \,
d\nu_{t}^{n,y},
\end{equation}
where $ i_n:H_n\times \R^r\rightarrow H\times \R^r$ denotes the inclusion map.  Thus, for each $n$
\begin{align*}
\int_{H_n\times\R^r} & |(f\circ i_n)(x)|\,d\nu_t^{n,y}(x)
= \int_{H_n\times\R^r} J^{n,y}_t(x)|(f\circ i_n)(x)|\,d\nu_t^{n}(x)
\\
& \leqslant \|f\circ i_n\|_{L^{q'}(H_n\times \R^r, \nu_t^n)}
	\left(\prod_{j=1}^r A_{j,q}(h,k)\right)\exp\left( \frac{1+q}{4t}\|h\|_{I^c}^2\right),
\end{align*}
where $q'$ is the conjugate exponent to $q$.
Combining this inequality with the limit in \eqref{e.5.7} implies that
\begin{multline}
\label{e.c}
	\int_{W\times\R^r} |f(x)|\,d\nu_t^y(x) \\
	\leqslant \|f\|_{L^{q'}(W\times\R^r,\nu_t)}   \left(\prod_{j=1}^r A_{j,q}(h,k)\right)\exp\left( \frac{1+q}{4t}\|h\|_{I^c}^2\right).
\end{multline}
Thus, we have proved that \eqref{e.c} holds for $f\in BC(W\times \R^r)$ and $y\in
\cup_{P\in\mathrm{Proj}(W)} PH\times\R^r$.  As this union is dense in $H\times \R^r$,
dominated convergence along with the continuity of the norm and inner product in $y$ implies
that \eqref{e.c} holds for all $y\in H\times \R^r$.

Since the bounded continuous functions are dense in $L^{q'}(W\times \R^r,\nu_t)$ (see for
example \cite[Theorem A.1]{JansonBook1997}) the inequality in \eqref{e.c} implies that, for all $t>0$ and $y=(h,k)\in H\times\R^r$, the linear functional $\varphi^y_t:BC(W\times\R^r)\rightarrow\mathbb{R}$ defined by
	\[ \varphi^y_t(f) = \int_{W\times\R^r} f(x)\,d\nu_t^y(x) \]
has a unique extension to an element of $L^{q'}(W\times\R^r,\nu_t)^*$, still denoted by
$\varphi^y_t$, which satisfies the bound
	\[ |\varphi^y_t(f)| \leqslant \|f\|_{L^{q'}(W\times\R^r,\nu_t)}
	\left(\prod_{j=1}^r A_{j,q}(h,k)\right)\exp\left( \frac{1+q}{4t}\|h\|_{I^c}^2\right)\]
for all $f\in L^{q'}(W\times\R^r,\nu_t)$.  Since $L^{q'}(W\times\R^r,\nu_t)^*\cong L^q(W\times\R^r,\nu_t)$, there
then exists a function $J^y_t\in
L^q(W\times\R,\nu_t)$ such that
\begin{equation}
\label{e.d}
	\varphi^y_t(f) = \int_{W\times\R^r} f(x)J^y_t(x)\,d\nu_t(x),
\end{equation}
for all $f\in L^{q'}(W\times\R^r,\nu_t)$, and
\[
\|J^y_t\|_{L^q(W\times\R^r,\nu_t)}
	\leqslant \left(\prod_{j=1}^r A_{j,q}(h,k)\right)\exp\left( \frac{1+q}{4t}\|h\|_{I^c}^2\right). \]
Now restricting \eqref{e.d} to $f\in BC(W\times\R^r)$, we may rewrite this equation as
\begin{equation}
\label{e.last}
	\int_{W\times\R^r} f(x)\,d\nu^y_t(x)
	= \int_{W\times\R^r} f(x) J^{y}_t(x)\,d\nu_t(x).
\end{equation}
Then a monotone class argument (again use
\cite[Theorem A.1 ]{JansonBook1997}) shows that \eqref{e.last} is valid for all
bounded measurable functions $f$ on $G$.  Thus,
$d\nu_t^y/d\nu_t$ exists and is given by $J^y_t$, which is in
$L^q(\nu_t)$ for all $q\in(1,\infty)$ and satisfies the
desired bound.
\end{proof}

The overall approach given above can be adapted to allow for functions $F$ taking values in infinite-dimensional spaces. As an example, we will consider the case that $F$ takes values in $W$, but the technique could be modified to allow for $F$ taking values in other infinite-dimensional spaces of interest. To begin, essentially the same proof of Proposition \ref{p.gapprox} yields the following modified convergence result for finite-dimensional approximations.

\begin{proposition}
	\label{p.gapproxH2} Let $\{P_n\}_{n=1}^\infty\subset \mathrm{Proj}(W)$ be any collection of projections such that $P_n|_H\uparrow I_H$. Suppose that $F:W\to W$ is continuous and there exists an orthonormal basis $\{h_j\}_{j=1}^\infty\subset H_*$ such that
\[ \sum_{j=1}^n \langle F(P_nB(t)),h_j\rangle h_j \to F(B(t))\]
a.s.\,\,in $W$.  Let $\{Q_n\}_{n=1}^\infty\subset \mathrm{Proj}(W)$ denote the sequence of projections associated to $\{h_j\}_{j=1}$ and consider
	\[ \tilde{Y}_n(t) := \left(P_nB(t), \int_0^t Q_nF(P_nB(s))\,ds\right) \]
	on $P_nH\times Q_nF(P_nH)$.
Then
	\[ \lim_{n\rightarrow\infty} \max_{0\leqslant t\leqslant T} \|Y(t)-\tilde{Y}_n(t)\|_{W\times W} =0 \text{ a.s.} \]
\end{proposition}

An example satisfying the hypotheses of Proposition \ref{p.gapproxH2} is the identity function $F(w)=w$ with $Q_n=P_n$, or slightly more generally, we could have $F(B)=\sum \langle B,h_j\rangle Ah_j$, where $A$ is a continuous linear operator on $H$; see Theorem 3.5.1 of \cite{BogachevGaussianMeasures}. Generally, if $\langle F(B_t),h_j\rangle h_j$ are independent symmetric random values in $W$ and $Q_nF(P_nB_t)\to F(B_t)$ in probability, then \cite[Proposition 2.11]{DaPratoZabczykBook1992} implies that $\sum_{j=1}^\infty \langle F(B_t),h_j\rangle h_j$ converges almost surely in $W$ to $F(B_t)$. We also provide the following example.

\begin{example}\label{ex.nonlin}
	Suppose that $F_k:=\langle F,h_k\rangle$ are cylinder functions
\[ F_k(x)=\phi_k(\langle w,e_k\rangle) \]
where $\phi_k:\mathbb{R}\to\mathbb{R}$ is of the form
	\[ \phi_k(x)=x + a_k\ln(|x|)1_{\{|x|> 1\}} \]
	with  $\{a_k\}\in\ell^2$. Note that
\begin{align*}
	\sum_{k=1}^\infty a_k \ln(|\langle B_t,h_k\rangle|)1_{\{|\langle B_t,h_k\rangle|\ge1\}} h_k
\end{align*}
	actually converges in $H$ (versus $W$) since for example, taking $B_t^k:=\langle B_t,h_k\rangle$ which is equal in distribution to a $\mathrm{Normal}(0,t)$ random variable $Z$ for all $k$, we have
\begin{multline*}
\mathbb{E}\left\|\sum_{k=1}^\infty a_k\ln(|B_t^k|)1_{\{|B_t^k|\geqslant 1\}} h_k\right\|_H^2
	= 2\sum_{k=1}^\infty a_k^2\mathbb{E}[(\ln B_t^k)^2 1_{\{B^k_t\geqslant 1\}}] \\
	= 2\sum_{k=1}^\infty a_k^2 \mathbb{E}[(\ln Z)^2 1_{\{Z \geqslant 1\}}]
	\leqslant C \sum_{k=1}^\infty a_k^2 <\infty.
\end{multline*}
	For clarity we have kept this example as concrete as possible but there are many generalizations, for example letting
\[ \phi_k(x)= 2x + a_kg(x)1_{\{|x|> 1\}} \]
	with $g(x)=-x+x^{-p}$ for any $p>-1$, or considering multivariate versions of $f_k$.
\end{example}

Allowing built-in scaling as in the previous example, there are many functions $F$ that one could consider (for example, $g$ could be  any function integrable with respect to the normal distribution so that $\phi_k$ is continuous), but we have included the specific choices above with a view toward the following assumption which allows us to prove the necessary functional inequalities for the finite-dimensional approximations.

\begin{assumption}\label{B4}
	Suppose $F:W\to W$ is a function satisfying the hypothesis of Proposition \ref{p.gapproxH2}.  Assume that $F_j:= \langle F,h_j\rangle$ is $H$-differentiable for all $j$ and that there exists an orthonormal basis $\{e_i\}_{i=1}^\infty\subset H_*$ of $H$, non-empty disjoint subsets $I_j \subset \mathbb{N}$ and constants $m_j,M_j>0$ such that, for each $j$ and for all $w\in W$,
\[
m_j \leqslant \langle \nabla F_j(w),e_i\rangle  \leqslant  M_j, \quad \text{ for all } i \in I_j
\]
and
\[
	\langle \nabla F_j(w),e_i\rangle=0, \quad \text{ for all } i\notin I_j.
\]
\end{assumption}

\begin{example}\label{ex.nonlin2}
Note that the functions from Example \ref{ex.nonlin} satisfy Assumption \ref{B4} with slight modifications for smoothness: let
\[ \phi_k(x)= cx + a_k\left(g_1(x)1_{\{1-\varepsilon\le|x|\leqslant1+\varepsilon\}} + g_2(x)1_{\{|x|> 1+\varepsilon\}}\right) \]
with $\varepsilon\in(0,1)$, with $c=1$ and $g_2(x)=\ln(|x|)$ or $c=2$ and $g_2(x)=-x+x^{-p}$ for any $p>-1$, and $g_1$ a smooth function such that $g_1(\pm(1-\varepsilon))=\pm c(1-\varepsilon)$, $g_1(\pm(1+\varepsilon))=g_2(\pm(1+\varepsilon)))$,
$g_1'(\pm(1-\varepsilon))=c$, and $g_1'(\pm(1+\varepsilon))=g_2'(\pm(1+\varepsilon))$. For the given functions, it is clearly possible to define such a $g_1$ so that $g_1'$ is bounded above and away from 0. \tai{I should probably just say that $g_1$ transitions smoothly from $cx$ to $g_2(x)$ around $x=1$.}
\end{example}

\begin{theorem}\label{t.gkqi4}
	Suppose that Assumption \ref{B4} holds for $F:W\to W$. Fix $h,k\in H$, and let $\nu_t^{(h,k)}$ denote the law of the process
\[
Y_t^{(h,k)} = \left(h+ B_t, \int_0^t F(B_s+h)\,ds +k\right).
\]
For $q\in(1,\infty)$ and each $j\in\mathbb{N}$, let
	\begin{multline*} A_{j,q}(h,k) \\: = \exp\left(  \frac{3(1+q)M_j}{m_j^3t^3}  \left( {\frac{m_j t}{2}}\sum_{i\in I_j} \langle h,e_i\rangle + \langle k,h_j\rangle \right)^2\right)
	\exp\left( \frac{(1+q)M_j}{4m_jt}\|h\|_{I_j}^2\right) \end{multline*}
	with $\{e_i\}_{i=1}^\infty,\{h_j\}_{j=1}^\infty\subset H_*$ the orthonormal bases, $I_j\subset\mathbb{N}$, and $m_j$ and $M_j$ the bounds coming from Assumption \ref{B4}.
If
	\begin{equation}\label{cond} \left(\prod_{j=1}^\infty A_{j,q}(h,k)\right) <\infty, \end{equation}
then $\nu_t^{(h,k)}$ is mutually absolutely continuous with $\nu_t:=\nu_t^0$ and
	\begin{align*}
		\left\|\frac{d\nu_t^{(h,k)}}{d\nu_t}\right\|_{L^q(W\times W,\nu_t)}
		&\leqslant 			\left(\prod_{j=1}^\infty A_{j,q}(h,k)\right)\exp\left( \frac{1+q}{4t}\|h\|_{I^c}^2\right)
\end{align*}
	where $I^c := \left(\cup_{i=1}^\infty I_j\right)^c$.
\end{theorem}

\begin{proof}
	Fix $t>0$, let $\{P_n\}_{n=1}^\infty\subset \mathrm{Proj}(W)$ be any collection of projections such that $P_n\uparrow I_H$. Let $\{h_j\}_{j=1}^\infty$ be the orthonormal basis
	as in Assumption \ref{B4} with each $F_j=\langle F,h_j\rangle$ $H$-differentiable, and let $\{Q_n\}_{n=1}^\infty$ be the associated sequence of projections.  Fix $n'\in\mathbb{N}$ and take $y=(h,k)\in P_{n'}H\times Q_{n'}F(P_{n'}H)$. For $n\geqslant n'$, we take
	\[ \tilde{Y}^y_n(t) := \left(h+ P_nB(t), \int_0^t Q_nF(P_nB(s)+h)\,ds + k\right) \]
	with $\nu_t^{n,y} = \mathrm{Law}(\tilde{Y}_n^y(t))$. Given Assumption \ref{B4}, one again applies Proposition \ref{p.gintharnack3} to show that for each $n$ the Radon-Nikodym derivative $\tilde{J}^{n,y}_t:=d\nu_t^{n,y}/d\nu_t^{n,0}$ satisfies
	\begin{align*}
		\|\tilde{J}^{n,y}_t\|_{L^q(P_nH\times Q_nF(P_nH),\nu_t^{n,0})}
	&\leqslant \left(\prod_{j=1}^n A_{j,q}(h,k)\right)\exp\left( \frac{1+q}{4t}\|h\|_{I^c}^2\right).
\end{align*}
Then in light of Proposition \ref{p.gapproxH2} (or rather its appropriate modification for $Y^y$), if the limit on the right hand side remains bounded as $n$ goes to $\infty$,  the proof follows as in Theorem \ref{t.gkqi3}.
\end{proof}

\begin{example}\label{ex.genex}
	If each $F_j$ is a cylinder function, that is, each $I_j$ is finite, and $m_j\geqslant m$ and $M_j\leqslant M$ for all $j$, then (\ref{cond}) holds if
	\[ \sum_{j=1}^\infty
		\left( \sum_{i\in I_j} \langle h,e_i\rangle + \langle k,h_j\rangle \right)^2 < \infty, \]
		which we can see (by expanding the square and applying Cauchy-Schwarz) is true when
		\[\sum_{j=1}^\infty \left( \sum_{i\in I_j} \langle h,e_i\rangle \right)^2< \infty \qquad \text{ and } \qquad \sum_{j=1}^\infty \langle k,h_j\rangle^2 <\infty.\]
	If we further assume that $\sup_j \# I_j\leqslant N$ then
	\[ \sum_{j=1}^\infty \left( \sum_{i\in I_j} \langle h,e_i\rangle \right)^2
	 \leqslant N \sum_{j=1}^\infty  \sum_{i\in I_j} \langle h,e_i\rangle^2 \leqslant N\|h\|_H^2 \]
	and (\ref{cond}) would hold for all $h,k\in H$. Thus, for example, for $F$ as in Example \ref{ex.nonlin2}, one has quasi-invariance for all $h,k\in H$.

	In the case that $F(w)=w$, we are back in the setting of Section \ref{s.infK} with a ``standard'' infinite-dimensional Kolmogorov diffusion as in Theorem \ref{t.kqi}, and Theorem \ref{t.gkqi4} actually yields a better bound than Theorem \ref{t.kqi} and is comparable with the path space $L^q$ norm obtained in Remark \ref{r.CMM}, improving that estimate for $q>3$. Here we can take $h_j=e_j$, and we have $I_j=\{j\}$ and $m_j=M_j=1$ for all $j$, which gives the bound
\begin{align*}
	\bigg\|\frac{d\nu_t^{(h,k)}}{d\nu_t}&\bigg\|_{L^q(W\times W,\nu_t)} \\
		&\leqslant \exp\left(\frac{3(1+q)}{t^3}\sum_j\left(\frac{t}{2}\langle h,e_j\rangle+\langle k,e_j\rangle\right)^2\right)\exp\left(\frac{1+q}{4t}\|h\|_H^2\right)\\
		&= \exp\left((1+q)\left(\frac{\|h\|_H^2}{t} + \frac{3\langle h,k\rangle}{t^2} + \frac{3\|k\|_H^2}{t^3} \right)\right).
	\end{align*}
\end{example}

\begin{remark}
As mentioned previously, these assumptions and the proof could be modified to consider functions $F$ taking values in other infinite-dimensional spaces. Another natural example in the present context would be $F$ taking values in $H$.

\end{remark}

\providecommand{\bysame}{\leavevmode\hbox to3em{\hrulefill}\thinspace}
\providecommand{\MR}{\relax\ifhmode\unskip\space\fi MR }
\providecommand{\MRhref}[2]{%
  \href{http://www.ams.org/mathscinet-getitem?mr=#1}{#2}
}
\providecommand{\href}[2]{#2}


\begin{thebibliography}{10}

\bibitem{BaudoinBonnefont2012}
Fabrice Baudoin and Michel Bonnefont, \emph{Log-{S}obolev inequalities for
  subelliptic operators satisfying a generalized curvature dimension
  inequality}, J. Funct. Anal. \textbf{262} (2012), no.~6, 2646--2676.
  \MR{2885961}

\bibitem{BaudoinFengGordina2019}
Fabrice Baudoin, Qi~Feng, and Maria Gordina, \emph{Integration by parts and
  quasi-invariance for the horizontal wiener measure on foliated compact
  manifolds}, Journal of Functional Analysis \textbf{277} (2019), no.~5, 1362
  -- 1422.

\bibitem{BaudoinGarofalo2017}
Fabrice Baudoin and Nicola Garofalo, \emph{Curvature-dimension inequalities and
  {R}icci lower bounds for sub-{R}iemannian manifolds with transverse
  symmetries}, J. Eur. Math. Soc. (JEMS) \textbf{19} (2017), no.~1, 151--219.
  \MR{3584561}

\bibitem{BaudoinGordinaMariano2020}
Fabrice Baudoin, Maria Gordina, and Phanuel Mariano, \emph{Gradient bounds for
  {K}olmogorov type diffusions}, Ann. Inst. H. Poincaré Probab. Statist.
  \textbf{56} (2020), no.~1, 612--636.

\bibitem{BaudoinGordinaMelcher2013}
Fabrice Baudoin, Maria Gordina, and Tai Melcher, \emph{Quasi-invariance for
  heat kernel measures on sub-{R}iemannian infinite-dimensional {H}eisenberg
  groups}, Trans. Amer. Math. Soc. \textbf{365} (2013), no.~8, 4313--4350.
  \MR{3055697}

\bibitem{BogachevGaussianMeasures}
Vladimir~I. Bogachev, \emph{Gaussian measures}, Mathematical Surveys and
  Monographs, vol.~62, American Mathematical Society, Providence, RI, 1998.
  \MR{MR1642391 (2000a:60004)}

\bibitem{DaPratoZabczykBook1992}
Giuseppe Da~Prato and Jerzy Zabczyk, \emph{Stochastic equations in infinite
  dimensions}, Encyclopedia of Mathematics and its Applications, vol.~44,
  Cambridge University Press, Cambridge, 1992. \MR{MR1207136 (95g:60073)}

\bibitem{DriverEldredgeMelcher2016}
Bruce~K. Driver, Nathaniel Eldredge, and Tai Melcher, \emph{Hypoelliptic heat
  kernels on infinite-dimensional {H}eisenberg groups}, Trans. Amer. Math. Soc.
  \textbf{368} (2016), no.~2, 989--1022. \MR{3430356}

\bibitem{DriverGordina2008}
Bruce~K. Driver and Maria Gordina, \emph{Heat kernel analysis on
  infinite-dimensional {H}eisenberg groups}, J. Funct. Anal. \textbf{255}
  (2008), no.~9, 2395--2461. \MR{MR2473262}

\bibitem{DriverGordina2009}
\bysame, \emph{Integrated {H}arnack inequalities on {L}ie groups}, J.
  Differential Geom. \textbf{83} (2009), no.~3, 501--550. \MR{MR2581356}

\bibitem{Gordina2017}
Maria Gordina, \emph{An application of a functional inequality to
  quasi-invariance in infinite dimensions}, pp.~251--266, Springer New York,
  New York, NY, 2017.

\bibitem{Gross1967a}
Leonard Gross, \emph{Potential theory on {H}ilbert space}, J. Functional
  Analysis \textbf{1} (1967), 123--181. \MR{0227747 (37 \#3331)}

\bibitem{Hormander1967a}
Lars H{\"o}rmander, \emph{Hypoelliptic second order differential equations},
  Acta Math. \textbf{119} (1967), 147--171. \MR{0222474 (36 \#5526)}

\bibitem{JansonBook1997}
Svante Janson, \emph{Gaussian {H}ilbert spaces}, Cambridge Tracts in
  Mathematics, vol. 129, Cambridge University Press, Cambridge, 1997.
  \MR{MR1474726 (99f:60082)}

\bibitem{Kolmogorov1934}
A.~Kolmogoroff, \emph{Zuf\"allige {B}ewegungen (zur {T}heorie der {B}rownschen
  {B}ewegung)}, Ann. of Math. (2) \textbf{35} (1934), no.~1, 116--117.
  \MR{1503147}

\bibitem{KuoLNM1975}
Hui~Hsiung Kuo, \emph{Gaussian measures in {B}anach spaces}, Springer-Verlag,
  Berlin, 1975, Lecture Notes in Mathematics, Vol. 463. \MR{MR0461643 (57
  \#1628)}

\bibitem{MarianoPhDThesis2018}
Phanuel Mariano, \emph{Functional inequalities for hypoelliptic diffusions
  using probabilistic and geometric methods}, Ph.D. thesis, University of
  {C}onnecticut, 2018.

\end{thebibliography}
\end{document}